\documentclass[11pt,twoside]{article}

\usepackage{amsmath}
\usepackage{amssymb}
\usepackage{amsthm}
\usepackage[dvips]{graphicx}
\usepackage{cases}

\newtheorem{thm}{Theorem}[section]

\newtheorem{lem}[thm]{Lemma}
\newtheorem{prop}[thm]{Proposition}

\newtheorem{rem}[thm]{Remark}

\DeclareMathOperator{\Ker}{Ker}

\title{\textbf{\Large Splitting theorem for sheaves of holomorphic $k$-vectors on complex contact manifolds}}

\author{\textsc{Takayuki MORIYAMA
and
Takashi NITTA
}}
\date{\today}

\begin{document}
\maketitle
\thispagestyle{empty}


\begin{quote}
\textbf{\fontsize{10pt}{15pt}\selectfont Abstract.}
\textrm{\fontsize{10pt}{15pt}\selectfont 
A complex contact structure $\gamma$ is defined by a system of holomorphic local $1$-forms satisfying the completely non-integrability condition. 
The contact structure induces a subbundle $\Ker \gamma$ of the tangent bundle and a line bundle $L$. 
In this paper, we prove that the sheaf of holomorphic $k$-vectors on a complex contact manifold 
splits into the sum of $\mathcal{O}(\bigwedge^{k} \Ker \gamma)$ and $\mathcal{O}(L\otimes \bigwedge^{k-1} \Ker \gamma)$ as sheaves of {\it $\mathbb{C}$-module}. 
The theorem induces the short exact sequence of cohomology of holomorphic $k$-vectors, 
and we obtain vanishing theorems for the cohomology of $\mathcal{O}(\bigwedge^{k} \Ker \gamma)$. 
}
\end{quote}

\section{Introduction}
Originating in physics, contact geometry is a mathematical formulation of classical mechanics. 
Contact geometry describes a geometric structure which appears in any constant energy hypersurface in the even-dimensional phase space of a mechanical system. 
In mathematics, the concept of contact structure appears explicitly in the work of Sophus Lie, and implicitly perhaps much earlier. 
By using a sheaf coefficient cohomology theory, Gray developed the idea and introduced a concept of {\it almost contact structure}~\cite{Gr}. 
He considered the deformation of a global contact structure in the terminology of homological algebra. 
Boothby and Wang studied the homogeneous manifolds associated with the contact transformation group~\cite{BW}. 
Furthermore, Kobayashi introduced the complex contact structure and developed several results of complex contact geometry~\cite{K1}. 
The complex contact structure is associated with a quaternionic structure with respect to the twistor correspondence \cite{I} \cite{Sa} \cite{Le}. 
From the beginning of the study of contact structures, it has been known that contact structures have a deep relationship with sheaf cohomology, for example, Gray'work. 

Let $M$ be a complex manifold of dimension $2n+1$. 
We consider a system $\{(U_i,\gamma_i) \}$ of an open covering $\{U_i\}$ of $M$ and holomorphic 1-forms $\gamma_i$ on $U_i$ 
such that $\gamma_i$ is a contact 1-form, that is, $(d\gamma_i)^n \wedge \gamma_i\neq 0$ on $U_i$, and $\gamma_i=f_{ij}\gamma_j$ for a holomorphic function $f_{ij}$ on $U_i\cap U_j$. 
We say that such systems $\{(U_i,\gamma_i )\}$ and $\{(U'_{i'},\gamma'_{i'} )\}$ are equivalent 
if there exists a holomorphic function $g_{ii'}$ on each intersection $U_i\cap U_{i'}$ so that $\gamma_{i}=g_{ii'}\gamma'_{i'}$ on $U_i\cap U_{i'}$, 
and call an equivalent class of $\{(U_i,\gamma_i) \}$ a {\it complex contact structure} on $M$. 
We denote by $\gamma=\{(U_i,\gamma_i) \}$ the contact structure on $M$. 
A pair $(M, \gamma)$ is called a \textit{complex contact manifold}. 
The contact structure $\gamma$ induces a line bundle $L$ on $M$ by the transition function $\{f_{ij}\}$. 
The canonical bundle $K_M$ is equal to $L^{-n-1}$ 
since $L^{n+1}\otimes \bigwedge^{2n+1}\mathbb{T}^* \cong \underline{\mathbb{C}}$ by the global section $(d\gamma_i)^{n+1}\wedge \gamma_i$.
The contact structure $\gamma$ is an $L$-valued $1$-form on $M$. 
We regard $\gamma$ as a bundle map from the holomorphic tangent bundle $\mathbb{T}$ of $M$ to $L$ 
and denote by $\Ker \gamma$ the kernel of the map $\gamma$ : 
\[
\Ker \gamma=\{v\in \mathbb{T}\mid \gamma(v)=0\}.
\]
There exists a short exact sequence of sheaves : 
\[
0\to \mathcal{O}(\Ker \gamma) \to \mathcal{O}(\mathbb{T}) \stackrel{\gamma}{\to} \mathcal{O}(L) \to 0. 
\]
Let $X$ be a holomorphic vector field on $M$. 
The Lie derivative $L_{X}\gamma_i$ of $\gamma_i$ with respect to $X$ is a $1$-form on $U_i$. 
The set $\{L_{X}\gamma_i \}$ is not a global form on $M$. 
However, the restriction $L_{X}\gamma_i|_{\Ker \gamma}$ to $\Ker \gamma$ defines a global section $\{L_{X}\gamma_i|_{\Ker \gamma} \}$ 
of the tensor $L\otimes\Ker \gamma^*$ of $L$ and the dual of $\Ker \gamma$. 
We call a vector field $X$ a {\it contact vector field} of $\gamma$ 
if $L_{X}\gamma_i|_{\Ker \gamma}=0$ for each $i$. 
The system of equations $\{L_{X}\gamma_i|_{\Ker \gamma}=0\}$ is a global and holomorphic equation for $X$. 
Such a vector field generates a contact automorphism. 
We define $\mathfrak{aut}(M, \gamma)$ to be the set of contact vector fields of $\gamma$. 
As an analogy of the real contact structures~\cite{K2}, Nitta and Takeuchi showed that 
for any element $s$ of $\mathcal{O}(L)$, 
there exists a unique contact vector field $X$ of $\gamma$ such that $\gamma(X)=s$. 
Moreover, the correspondence $\mathcal{O}(L) \to \mathfrak{aut}(M, \gamma)$ is isomorphic~\cite{NT}. 
It means that the holomorphic tangent sheaf $\mathcal{O}(\mathbb{T})$ splits into 
$\mathcal{O}(\Ker \gamma)$ and $\mathcal{O}(L)$ as sheaves of $\mathbb{C}$-module. 
LeBrun showed that $\mathcal{O}(\mathbb{T})$ does not split into the sum of $\mathcal{O}(\Ker \gamma)$ and $\mathcal{O}(L)$ 
as sheaves of $\mathcal{O}$-module on Fano manifolds~\cite{Le}. 
Therefore, the splitting of $\mathcal{O}(\mathbb{T})$ does not directly induce that of the sheaf $\mathcal{O}(\bigwedge^k \mathbb{T})$ of $k$-vectors. 

In this paper, we show the splitting of $\mathcal{O}(\bigwedge^k \mathbb{T})$ into the sum of 
$\mathcal{O}(\bigwedge^{k} \Ker \gamma)$ and $\mathcal{O}(L\otimes \bigwedge^{k-1} \Ker \gamma)$ as sheaves of $\mathbb{C}$-module. 
We extend the map $\gamma$ to a bundle map from $\bigwedge^k \mathbb{T}$ to $L\otimes \bigwedge^{k-1} \Ker \gamma$ by 
\[
v_{1}\wedge\dots \wedge v_{k}\mapsto \sum_{j=1}^k (-1)^{j-1}\gamma(v_{j})\otimes v_{1}\wedge\dots\wedge v_{j-1}\wedge v_{j+1} \wedge\dots \wedge v_{k}
\]
and denote the map also by $\gamma$ for simplicity. 
Since the kernel of the map $\gamma$ is just the space $\bigwedge^k \Ker \gamma$, we obtain the short exact sequence of the sheaves : 
\[
0\to \mathcal{O}(\bigwedge^k \Ker \gamma) \to \mathcal{O}(\bigwedge^k \mathbb{T}) \stackrel{\gamma}{\to} \mathcal{O}(L\otimes \bigwedge^{k-1}\Ker \gamma) \to 0.
\]
We shall extend to the equation $L_{X}\gamma_i|_{\Ker \gamma}=0$ for a $1$-vector $X$ to a global and holomorphic equation for $k$-vectors. 
Let $\nabla$ be a connection of the line bundle $L$ such that $\nabla^{0,1}=\bar{\partial}$. 
For a $1$-vector $X$, the local equation $L_{X}\gamma_i|_{\Ker \gamma}=0$ is given by 
the global equation $(d^{\nabla}\gamma )(X)+d^{\nabla}(\gamma (X))|_{\Ker \gamma}=0$ which is holomorphic. 
The direct extension $(\otimes^k d^{\nabla}\gamma)(X)+d^{\nabla}(\otimes^{k-1} d^{\nabla}\gamma(\gamma(X)))|_{\Ker \gamma}=0$ for a $k$-vector $X$ 
is global but not holomorphic whenever $\nabla$ is holomorphic. 
In order to find a global and holomorphic equation for holomorphic k-vector fields, 
we decompose the space $\bigwedge^k \Ker \gamma$ as a sum of primitive parts with respect to the symplectic structure on $\Ker \gamma$. 
Then we obtain such a equation which induces the following splitting theorem for $\mathcal{O}(\bigwedge^k \mathbb{T})$ : 
\begin{thm}\label{s1t2}
Let $k$ be an integer from $1$ to $2n+1$. 
The sequence 
\[
0\to \mathcal{O}(\bigwedge^k \Ker \gamma) \to \mathcal{O}(\bigwedge^k \mathbb{T}) \to \mathcal{O}(L\otimes \bigwedge^{k-1} \Ker \gamma) \to 0
\] 
splits as sheaves of $\mathbb{C}$-module. In particular, the sequence 
\[
0\to H^i(\bigwedge^k \Ker \gamma) \to H^i(\bigwedge^k \mathbb{T}) \to H^i(L\otimes \bigwedge^{k-1} \Ker \gamma) \to 0
\] 
is exact for each $i=0,\dots,2n+1$. 
\end{thm} 
We generalize the theorem to the splitting of $\mathcal{O}(L^{m}\otimes \bigwedge^k \mathbb{T})$ 
and the exact sequence of $H^i(L^{m}\otimes \bigwedge^k \mathbb{T})$ under a condition for $m$ and $k$ (see Theorem~\ref{s3t2}). 
As an application, we obtain the following vanishing theorem for $H^i(L^{m}\otimes \bigwedge^k \Ker \gamma)$ on compact K\"ahler manifolds 
by Kodaira-Akizuki-Nakano vanishing theorem : 
\begin{thm}\label{s1t4} 
If $M$ is a compact K\"ahler complex contact manifold with $c_1(M)>0$, 
then $H^i(M, L^m \bigwedge^{k}  \Ker \gamma)=\{0\} $ for $k$ and $m$ satisfying one of following three conditions 
\begin{equation*}
\quad 
\left\{
\begin{array}{lll}
i\le 2n-k,& m \le -[\frac{k+1}{2}]-n-1,& 1\le k\le 2n+1, \\
\forall i,& -n\le m \le -k-1,& 1\le k\le n-1, \\
i\ge k+1,& m \ge -[\frac{k}{2}],&  1\le k\le 2n+1. 
\end{array}
\right.
\end{equation*}
\end{thm}
We also have a similar result for vanishing of cohomology in the case of $c_1(M)<0$ (see Theorem~\ref{s4.1t1}). 
Moreover, on $\mathbb{CP}^{2n+1}$, we show the vanishing theorem for $H^i(\mathcal{O}(m) \bigwedge^k \Ker \gamma)$ 
by Bott vanishing theorem (see Theorem~\ref{s4.2t2}).

This paper is organized as follows. 
In Section 2, we prepare some propositions for complex symplectic vector spaces. 
The space of $k$-vectors is decomposed into a sum of primitive parts with respect to the symplectic structure. 
In Section 3, the equation $L_{X}\gamma_i|_{\Ker \gamma}=0$ for a $1$-vector $X$ is extended to a global and holomorphic equation for $k$-vectors. 
By using the equation, we prove Theorem~\ref{s1t2} (see Theorem~\ref{s3t1}). 
The theorem is generalized to the splitting theorem for $\mathcal{O}(L^{m}\otimes \bigwedge^k \mathbb{T})$ (Theorem~\ref{s3t2}). 
On $\mathbb{CP}^{2n+1}$, $L$ is given by $\mathcal{O}(2)$ and we also obtain the splitting of $\mathcal{O}(m)\otimes \bigwedge^k \Ker \gamma$. 
In the last section, we show two kinds of vanishing theorems (Theorem~\ref{s4.1t1} and Theorem~\ref{s4.2t2}). 

\section{Complex symplectic vector spaces}
In this section, we prepare some propositions for complex symplectic vector spaces 
in order to show our main theorems in Section 3. 

\subsection{Symplectic structures}
Let $V$ be a complex vector space of dimension $2n$. 
A {\it complex symplectic vector space} is a pair $(V, \omega)$ of $V$ 
and a non-degenerate skew-symmetric bilinear form $\omega$ on $V$. 
If we regard the symplectic structure $\omega$ as the isomorphism from $V$ to the dual space $V^*$, 
then the $2$-tensor $\otimes^2 \omega$ of $\omega$ is the isomorphism 
\[
\otimes^2 \omega : \otimes^2 V \to \otimes^2 V^*.
\]
It induces the isomorphism $\otimes^2 \omega : \wedge^2 V \to \wedge^2 V^*$. 
A $2$-vector $w_0$ is defined by  
\[
\otimes^2 \omega(w_0)=\omega. 
\]
If we take a basis $e_1,\dots,e_{2n}$ of $V$ such that $\omega=e_1^*\wedge e_2^*+\dots+e_{2n-1}^*\wedge e_{2n}^*$, 
then $w_0$ is represented as $w_0=e_1\wedge e_2+\dots+e_{2n-1}\wedge e_{2n}$ and $\omega(w_0)=n$.

\subsection{Decomposition of the space of $k$-vectors}
Let $k$ be an integer from $1$ to $2n$. 
We define an operator $\mathcal{L}: \wedge^k V \to\wedge^{k+2}V$ by $\mathcal{L}(X)=X\wedge w_0$ for $X\in \wedge^k V$. 
We consider a $1$-form $\theta$ as a map from $\wedge^k V$ to $\wedge^{k-1} V$ by 
\[
\theta(v_{1}\wedge\dots \wedge v_{k})=\sum_{j=1}^k (-1)^{j-1}\theta(v_{j}) v_{1}\wedge\dots\wedge v_{j-1}\wedge v_{j+1} \wedge\dots \wedge v_{k}
\]
for $v_{1},\dots, v_{k} \in V$ and denote the map also by $\theta$ for simplicity. 
Let $l$ be an integer with $0\le l\le k$. 
We regard an $l$-form $\theta_1\wedge \dots \wedge \theta_l$ as 
a map $\theta_1\wedge \dots \wedge \theta_l : \wedge^k V \to \wedge^{k-l} V$ by 
\[
\theta_1\wedge \dots \wedge \theta_l(X)= \theta_l\circ \dots \circ \theta_2\circ \theta_1(X)
\]
for $X\in \wedge^k V$. 
Let $\Lambda$ be an operator $\Lambda : \wedge^k V \to \wedge^{k-2} V$ defined by $\Lambda(X)= \omega(X)$ for $X\in \wedge^k V$. 
Then we obtain the formula $(\Lambda \mathcal{L}-\mathcal{L}\Lambda)(X)=(n-k)X$ for $X\in \wedge^k V$, 
and inductively,  
\begin{equation}
(\Lambda \mathcal{L}^r-\mathcal{L}^r\Lambda)(X)=r(n-k-r+1)\mathcal{L}^{r-1}X \label{s2eq2}
\end{equation}
where we define $\mathcal{L}^0={\rm id}$. 
A $k$-vector $X$ is called {\it primitive} if $\Lambda(X)=0$. 
If follows from the formula (\ref{s2eq2}) that 
\begin{equation}
\Lambda^s \mathcal{L}^r X=
\left\{
\begin{array}{ll}
\frac{r!}{(r-s)!}(n-k-r+1)(n-k-r+2)\cdots(n-k-r+s)\mathcal{L}^{r-s}X, &\quad r\ge s, \\
0 , &\quad r<s
\end{array}
\right. \label{s2eq3}
\end{equation}
for a primitive $k$-vector $X$. 
Let $\wedge_e^k V$ denote the space of primitive $k$-vectors : 
\[
\wedge_e^k V=\{X\in \wedge^k V \mid \Lambda(X)=0 \}. 
\] 
Then we have the following decomposition of the space $\wedge^k V$ of $k$-vectors : 
\begin{prop} \label{s2p1} 
If $k\le n$, then 
\[
\wedge^k V=\wedge_e^k V+\mathcal{L}\wedge_e^{k-2} V+\dots+ \mathcal{L}^{[\frac{k}{2}]}\wedge_e^{k-2[\frac{k}{2}]} V.
\]
If $k> n$, then 
\[
\wedge^k V=\mathcal{L}^{k-n}\wedge_e^{2n-k} V+\mathcal{L}^{k-n+1}\wedge_e^{2n-k-2} V+\dots+ \mathcal{L}^{[\frac{k}{2}]}\wedge_e^{k-2[\frac{k}{2}]} V
\]
where $[m]$ means the Gauss symbol of $m$. $\hfill\Box$
\end{prop}

\subsection{Transformation associated with the decomposition}
We define a linear transformation $T$ on $\wedge^k V$ by 
\[
T(X)=c_0X+c_1\mathcal{L}\Lambda X+c_2\mathcal{L}^2\Lambda^2 X+\dots+c_{[\frac{k}{2}]}\mathcal{L}^{[\frac{k}{2}]}\Lambda^{[\frac{k}{2}]}X=\sum_{i=0}^{[\frac{k}{2}]}c_i\mathcal{L}^i\Lambda^i X
\]
for constants $c_0, c_1, c_2,\dots, c_{[\frac{k}{2}]}$, then we obtain the following : 
\begin{prop}\label{s2p2} 
The transformation $T$ is isomorphic if and only if the constants $c_0, c_1,\dots, c_{[\frac{k}{2}]}$ satisfy 
$\sum_{s=0}^{r} c_s \frac{r!}{(r-s)!}\frac{(n-k+r+s)!}{(n-k+r)!}\neq 0$ for any $r=0,\dots,[\frac{k}{2}]$. 
\end{prop}
\begin{proof} 
In the case $k\le n$ and $0\le r \le [\frac{k}{2}]$, 
the equation (\ref{s2eq3}) implies that  
\[
\mathcal{L}^s\Lambda^s \mathcal{L}^r X=
\left\{
\begin{array}{ll}
\frac{r!}{(r-s)!}\frac{(n-k+r+s)!}{(n-k+r)!}\mathcal{L}^r X, &\quad s\le r, \\
0 , &\quad s>r 
\end{array}
\right.
\]
for $X\in \wedge_e^{k-2r}$. 
It yields that 
\begin{equation}
T( \mathcal{L}^r X)=\sum_{s=0}^{ [\frac{k}{2}]} c_s \mathcal{L}^s\Lambda^s \mathcal{L}^r X=\left(\sum_{s=0}^{r} c_s \frac{r!}{(r-s)!}\frac{(n-k+r+s)!}{(n-k+r)!}\right)\mathcal{L}^r X
\label{s2eq4}
\end{equation}
for $X\in \wedge_e^{k-2r}$. 
In the case $k>n$ and $k-n\le r \le [\frac{k}{2}]$, we also have the same equation (\ref{s2eq4}) for $X\in \wedge_e^{k-2r}$. 
Hence Proposition~\ref{s2p1} implies that $T$ is an isomorphism from $\wedge^k V$ to itself 
if and only if $\sum_{s=0}^{r} c_s \frac{r!}{(r-s)!}\frac{(n-k+r+s)!}{(n-k+r)!}\neq 0$ for each $r=0,\dots,[\frac{k}{2}]$. 
\end{proof} 


\section{Splitting of sheaves on complex contact manifolds}
Let $(M, \gamma)$ be a complex contact manifold of dimension $2n+1$ and $L$ the line bundle associated with the contact structure $\gamma$. 
We denote by $D$ the subbundle $\Ker \gamma$ of $\mathbb{T}$. 
Let $\nabla$ be a connection of $L$ with $\nabla^{0,1}=\bar{\partial}$. 
The covariant exterior differentiation $d^{\nabla}\gamma$ of $\gamma$ is a smooth section of $L\otimes \bigwedge^2 \mathbb{T}^*$.

\subsection{The isomorphisms associated with $d^{\nabla}\gamma$}
Let $d^{\nabla}\gamma|_D$ denote the restriction of $d^{\nabla}\gamma$ to $D$. 
Then, $d^{\nabla}\gamma|_D$ is a holomorphic section of $L\otimes \bigwedge^2 D^*$ which is independent of the choice of a connection $\nabla$. 
We identify $d^{\nabla}\gamma|_D$ with the holomorphic bundle map from $D$ to $L\otimes D^*$. 
Then the map $d^{\nabla}\gamma|_D: D\to L\otimes D^*$ is isomorphic since $d\gamma_i$ is non-degenerate on $D$. 
In general, the $k$-th tensor $\otimes^k d^{\nabla}\gamma$ of $d^{\nabla}\gamma$ is a smooth section of $L^k\otimes^{k}\bigwedge^2 \mathbb{T}^*$ 
which is regarded as the smooth bundle map 
\begin{equation*}
\otimes^k d^{\nabla}\gamma : \bigwedge^k \mathbb{T} \to L^k\otimes\bigwedge^k \mathbb{T}^*.\label{s3eq1}
\end{equation*}
Let $e$ be a local frame of $L$ and $A$ a connection form of $\nabla$ with respect to $e$. 
The contact structure $\gamma$ is given by $\gamma=e\otimes \gamma_0$ for a holomorphic $1$-form $\gamma_0$. 
Then $d^{\nabla}\gamma=e\otimes (d\gamma_0 +A\wedge \gamma_0)$. 
For $1$-vectors $v_1,\dots , v_k$, 
the $L^k$-valued $k$-form $\otimes^k d^{\nabla}\gamma(v_1\wedge \cdots \wedge v_k)$ is written by
\begin{eqnarray*}
&&e^k\otimes \{(d\gamma_0 +A\wedge \gamma_0)(v_1)\wedge\cdots \wedge(d\gamma_0 +A\wedge \gamma_0)(v_k)\} \\
&=&e^k\otimes \{\otimes^{k} d\gamma_0(v_1 \wedge\cdots \wedge v_k)+\gamma_0\wedge (\otimes^{k-1}d\gamma_0)(A(v_1\wedge \cdots \wedge v_k))\\
&&-A\wedge (\otimes^{k-1}d\gamma_0)(\gamma_0(v_1\wedge \cdots \wedge v_k)) \}. 
\end{eqnarray*}
It implies 
\begin{equation}
\otimes^k d^{\nabla}\gamma(X)=e^k\otimes \{\otimes^k d\gamma_0(X)+\gamma_0 
\wedge (\otimes^{k-1}d\gamma_0)(A(X))-A\wedge (\otimes^{k-1}d\gamma_0)(\gamma_0(X))\} \label{s3eq3}
\end{equation}
for any $k$-vector $X$. 
We remark that the map $\otimes^k d^{\nabla}\gamma$ is not holomorphic whenever $\nabla$ is holomorphic. 
However, the restriction $\otimes^k d^{\nabla}\gamma|_D$ is the holomorphic map from $\bigwedge^k D$ to $L^k\otimes\bigwedge^k D^*$. 
The induced map $\otimes^k d^{\nabla}\gamma|_D: \mathcal{O}(\otimes \bigwedge^k D) \to \mathcal{O}(L^{k}\otimes\bigwedge^k D^*)$ is isomorphic on sheaves, 
and it is extended to an isomorphism 
\begin{equation*}
\otimes^k d^{\nabla}\gamma|_D : \mathcal{O}(L^m\otimes \bigwedge^k D) \to \mathcal{O}(L^{k+m}\otimes\bigwedge^k D^*) \label{s3eq2}
\end{equation*}
for each $m\in\mathbb{Z}$ and $k=1,\dots,2n+1$. 

In the case $m=-1$ and $k=2$, we have the isomorphism $\otimes^2 d^{\nabla}\gamma|_D : \mathcal{O}(L^{-1}\otimes \bigwedge^2 D) \to \mathcal{O}(L\otimes\bigwedge^2 D^*)$. 
By considering $d^{\nabla}\gamma|_D$ as a holomorphic section of $L\otimes\bigwedge^2 D^*$, 
there exists a holomorphic section $w$ of $L^{-1}\otimes \wedge^2 D$ such that 
\[
\otimes^2 d^{\nabla}\gamma|_D(w)=d^{\nabla}\gamma|_D 
\]
as $w_0$ in Section 2.1. 
The section $w$ is independent of the choice of a connection $\nabla$. 

\subsection{The map $F:\bigwedge^k \mathbb{T} \to L^k\otimes \bigwedge^k \mathbb{T}^*$}
We fix an integer $k$ such that $1\le k\le 2n+1$ and define a map $F:\bigwedge^k \mathbb{T} \to L^k\otimes \bigwedge^k \mathbb{T}^*$ by 
\[
F(X)=\sum_{i=0}^{[\frac{k}{2}]}\frac{(k-i)!}{k!i!}\otimes^{k-2i} d^{\nabla}\gamma((d^{\nabla}\gamma)^i(X))\wedge (d^{\nabla}\gamma)^i
\]
for $X\in \bigwedge^k \mathbb{T}$, where $(d^{\nabla}\gamma)^0$ is the identity map ${\rm id}$ and 
$(d^{\nabla}\gamma)^i$ is the $i$-th wedge $d^{\nabla}\gamma\wedge \dots\wedge d^{\nabla}\gamma$ of $d^{\nabla}\gamma$. 
We have 
\begin{equation}
(d^{\nabla}\gamma)^i(X)=e^i\otimes \{(d\gamma_0)^i(X)-i(d\gamma_0)^{i-1}(A(\gamma_0(X)))\} \label{s3eq4}
\end{equation}
for each $i=0,\dots ,[\frac{k}{2}]$ since $(d^{\nabla}\gamma)^i=e^i\otimes \{(d\gamma_0)^i+i A\wedge\gamma_0 \wedge (d\gamma_0)^{i-1} \}$. 
The equations (\ref{s3eq3}) and (\ref{s3eq4}) imply that 
\begin{eqnarray}
&&\otimes^{k-2i} d^{\nabla}\gamma((d^{\nabla}\gamma)^i(X)) \notag\\
&=&e^{k-i}\otimes \{ \otimes^{k-2i} d\gamma_0((d\gamma_0)^i(X))-i\otimes^{k-2i} d\gamma_0((d\gamma_0)^{i-1}(A(\gamma_0(X)))) \notag\\
&&+\gamma_0\wedge (\otimes^{k-1-2i}d\gamma_0)((d\gamma_0)^i(A(X)))-A\wedge (\otimes^{k-1-2i}d\gamma_0)((d\gamma_0)^i(\gamma_0(X))) \} \quad \label{s3eq5}
\end{eqnarray}
where we use $A((d\gamma_0)^{i-1}(A(\gamma_0(X))))=\gamma_0((d\gamma_0)^{i-1}(A(\gamma_0(X)))=0$ in the first line. 
It yields that 
\begin{eqnarray}
\otimes^{k-2i} d^{\nabla}\gamma((d^{\nabla}\gamma)^i(X))|_D
&=&e^{k-i}\otimes \{ \otimes^{k-2i} d\gamma_0((d\gamma_0)^i(X))-i\otimes^{k-2i} d\gamma_0((d\gamma_0)^{i-1}(A(\gamma_0(X)))) \notag\\
&&-A\wedge (\otimes^{k-1-2i}d\gamma_0)((d\gamma_0)^i(\gamma_0(X))) \}. \label{s3eq6}
\end{eqnarray}
The second term in the right hand side of the equation (\ref{s3eq6}) is $0$ except for $1\le i\le [\frac{k}{2}]$ and 
the third term is $0$ except for $0\le i\le [\frac{k-1}{2}]$. Hence 
\begin{eqnarray}
F(X)|_D
&=&\sum_{i=0}^{[\frac{k}{2}]}\frac{(k-i)!}{k!i!}\{ \otimes^{k-2i} d^{\nabla}\gamma((d^{\nabla}\gamma)^i(X))\wedge (d^{\nabla}\gamma)^i \}|_D \notag\\
&=&e^{k}\otimes \Big\{ \sum_{i=0}^{[\frac{k}{2}]}\frac{(k-i)!}{k!i!}\{ \otimes^{k-2i} d\gamma_0((d\gamma_0)^i(X)) \wedge (d\gamma_0)^i \notag\\
&&-\sum_{i=1}^{[\frac{k}{2}]}\frac{(k-i)!}{k!(i-1)!}\otimes^{k-2i} d\gamma_0((d\gamma_0)^{i-1}(A(\gamma_0(X)))) \wedge (d\gamma_0)^i \notag\\
&&-\sum_{i=0}^{[\frac{k-1}{2}]}\frac{(k-i)!}{k!i!}A\wedge (\otimes^{k-1-2i}d\gamma_0)((d\gamma_0)^i(\gamma_0(X))) \wedge (d\gamma_0)^i \Big\} \Big|_D  \label{s3eq6.5}.
\end{eqnarray}
The restriction $F|_{D}:\bigwedge^k D \to L^k\bigwedge^k D^*$ satisfies $F|_{D}(X)=F(X)|_{D}$ for $X\in \bigwedge^k D$. 
If $X$ is a holomorphic section of $\bigwedge^k D$, then it follows from $\gamma_0(X)=0$ that 
\[
F|_D(X)=F(X)|_D = e^{k}\otimes \Big\{ \sum_{i=0}^{[\frac{k}{2}]}\frac{(k-i)!}{k!i!} \otimes^{k-2i} d\gamma_0((d\gamma_0)^i(X)) \wedge (d\gamma_0)^i\Big\}  \Big|_D
\]
is a holomorphic section of $L^k\otimes \bigwedge^k D^*$. 
Hence we obtain the following : 
\begin{lem}\label{s3l1} 
An $L^k$-valued $k$-form $F|_{D}(X)$ restricted to $D$ is holomorphic for any holomorphic section $X$ of $\bigwedge^k D$. $\hfill\Box$
\end{lem}

The lemma implies that $F|_D$ is regarded as a map from $\mathcal{O}(\bigwedge^k D)$ to $\mathcal{O}(L^k\otimes \bigwedge^k D^*)$. 
The map $F|_{D}$ is written by 
\[
F|_{D}(X)=\otimes^k d^{\nabla}\gamma|_D\Big(\sum_{i=0}^{[\frac{k}{2}]}\frac{(k-i)!}{k!i!}(d^{\nabla}\gamma|_D)^i(X)\wedge (w)^i\Big)
\]
for $X\in \mathcal{O}(\bigwedge^k D)$ 
since $(d^{\nabla}\gamma|_D)^i=\otimes^{2i} d^{\nabla}\gamma|_D(w)^i$. 
We define the transformation $f: \mathcal{O}(\bigwedge^k D)\to \mathcal{O}(\bigwedge^k D)$ by  
\[
f(X)=\sum_{i=0}^{[\frac{k}{2}]}\frac{(k-i)!}{k!i!}(d^{\nabla}\gamma|_D)^i(X)\wedge (w)^i
\]
for $X\in \mathcal{O}(\bigwedge^k D)$. 
Then the map $F|_{D}$ is the composition of $\otimes^k d^{\nabla}\gamma|_D$ and $f$. 
Proposition~\ref{s2p2} implies that the following : 
\begin{prop}\label{s3p1} 
The map $F|_{D}:\mathcal{O}(\bigwedge^k D) \to \mathcal{O}(L^k\otimes \bigwedge^k D^*)$ is isomorphic. 
\end{prop}
\begin{proof} 
It suffices to show that $f$ is isomorphic since $\otimes^k d^{\nabla}\gamma|_D$ is isomorphic. 
At each point $x\in M$, the linear map $f_x:\bigwedge^k D_x\to \bigwedge^k D_x$ is written by 
\[
f_x=\sum_{i=0}^{ [\frac{k}{2}]} \frac{(k-i)!}{k!i!} \mathcal{L}^i\Lambda^i 
\]
where $\mathcal{L}$ and $\Lambda$ are operators as in the previous section associated with the symplectic structure $(d^{\nabla}\gamma|_D)_x$ on $D_x$. 
The map $f_x$ is isomorphic by Proposition~\ref{s2p2} since each coefficient $c_i=\frac{(k-i)!}{k!i!}$ is positive. 
Hence $F|_{D}$ is also isomorphic. 
\end{proof}

\subsection{The map $G:\Gamma(L\otimes\bigwedge^{k-1} D) \to \Gamma(L^k\otimes\bigwedge^k \mathbb{T}^*)$}
We define a map $G:\Gamma(L\otimes\bigwedge^{k-1} D) \to \Gamma(L^k\otimes\bigwedge^k \mathbb{T}^*)$ by 
\[
G(s)=\sum_{i=0}^{[\frac{k-1}{2}]}\frac{(k-1-i)!}{k!i!}d^{\nabla}\left( \otimes^{k-1-2i}d^{\nabla}\gamma((d^{\nabla}\gamma)^i(s))\right)\wedge (d^{\nabla}\gamma)^i
\]
for $s\in \Gamma(L\otimes\bigwedge^{k-1}D)$. 
Let $s$ be a section of $L\otimes\bigwedge^{k-1}D$. 
The section $s$ is locally written as $s=e\otimes s_0$ for a section $s_0\in \bigwedge^{k-1}D$. 
It follows from the equations (\ref{s3eq5}) and $(d^{\nabla}\gamma)^i(s)=e^{i+1}\otimes (d\gamma_0)^i(s_0)$ that 
\begin{equation}
\otimes^{k-1-2i} d^{\nabla}\gamma((d^{\nabla}\gamma)^i(s))
=e^{k-i}\otimes \{ \otimes^{k-1-2i} d\gamma_0((d\gamma_0)^i(s_0)+\gamma_0\wedge (\otimes^{k-2-2i}d\gamma_0)((d\gamma_0)^i(A(s_0))) \}. \label{s3eq7}
\end{equation}
We remark that the second term in the right hand side of the equation (\ref{s3eq7}) is $0$ except for $0\le i\le [\frac{k}{2}]-1$. 
It yields that 
\begin{eqnarray*}
d^{\nabla} \otimes^{k-1-2i} d^{\nabla}\gamma((d^{\nabla}\gamma)^i(s))|_D
&=&e^{k-i}\otimes \{ d \otimes^{k-1-2i} d\gamma_0((d\gamma_0)^i(s_0) \\
&&+d\gamma_0\wedge (\otimes^{k-2-2i}d\gamma_0)((d\gamma_0)^i(A(s_0)))\\
&& +(k-i)A \wedge \otimes^{k-1-2i} d\gamma_0((d\gamma_0)^i(s_0)) \} |_D. 
\end{eqnarray*}
Hence 
\begin{eqnarray}
G(s)|_D&=&\sum_{i=0}^{[\frac{k-1}{2}]}\frac{(k-1-i)!}{k!i!} d^{\nabla} (\otimes^{k-1-2i} d^{\nabla}\gamma((d^{\nabla}\gamma)^i(s))|_D) \wedge (d^{\nabla}\gamma)^i|_D \notag\\
&=&e^{k}\otimes \Big\{ \sum_{i=0}^{[\frac{k-1}{2}]}\frac{(k-1-i)!}{k!i!} d (\otimes^{k-1-2i} d\gamma_0((d\gamma_0)^i(s_0)))\wedge (d\gamma_0)^i  \notag\\
&&+\sum_{i=1}^{[\frac{k}{2}]}\frac{(k-i)!}{k!(i-1)!} (\otimes^{k-2i}d\gamma_0)((d\gamma_0)^{i-1}(A(s_0)))\wedge (d\gamma_0)^i  \notag\\
&& +\sum_{i=0}^{[\frac{k-1}{2}]}\frac{(k-i)!}{k!i!}A \wedge \otimes^{k-1-2i} d\gamma_0((d\gamma_0)^i(s_0))\wedge (d\gamma_0)^i \Big\} \Big|_D. \label{s3eq8}
\end{eqnarray}
Then we have 
\begin{prop}\label{s3p2} 
An $L^k$-valued $k$-form $\{F(X)+G(\gamma(X))\}|_{D}$ restricted to $D$ is holomorphic for any holomorphic $k$-vector $X$. 
\end{prop}
\begin{proof} 
If we take $s$ as the image $\gamma(X)$ of a $k$-vector $X$, then $s_0=\gamma_0(X)$, 
and it follows from the equations (\ref{s3eq6.5}) and (\ref{s3eq8}) that 
\begin{eqnarray*}
\{F(X)+G(s)\}|_{D}&=&e^{k}\otimes \Big\{ \sum_{i=0}^{[\frac{k}{2}]}\frac{(k-i)!}{k!i!}\{ \otimes^{k-2i} d\gamma_0((d\gamma_0)^i(X))\wedge (d\gamma_0)^i \\
&&+\sum_{i=0}^{[\frac{k-1}{2}]}\frac{(k-1-i)!}{k!i!} d(\otimes^{k-1-2i} d\gamma_0((d\gamma_0)^i(\gamma_0(X))))\wedge (d\gamma_0)^i \Big\} \Big|_D. 
\end{eqnarray*}
Hence $\{F(X)+G(\gamma(X))\}|_{D}$ is holomorphic for a holomorphic $k$-vector $X$. 
\end{proof} 

\subsection{The splitting of sheaves $\mathcal{O}(\bigwedge^k\mathbb{T})$ as $\mathbb{C}$-module}

We have the following theorem:
\begin{thm}\label{s3t1} 
Let $k$ be an integer from $1$ to $2n+1$. 
The sequence 
\[
0\to \mathcal{O}(\bigwedge^k D) \to \mathcal{O}(\bigwedge^k \mathbb{T}) \to \mathcal{O}(L\otimes \bigwedge^{k-1} D) \to 0
\] 
splits as sheaves of $\mathbb{C}$-module. 
In particular, the sequence 
\[
0\to H^i(\bigwedge^k D) \to H^i(\bigwedge^k \mathbb{T}) \to H^i(L\otimes \bigwedge^{k-1} D) \to 0
\] 
is exact for each $i=0,\dots,2n+1$. 
\end{thm}
\begin{proof} 
Let $s$ be a holomorphic section of $L\otimes \bigwedge^{k-1} D$. 
We take an open set $U$ of $M$ where the bundles $L$ and $D$ are trivial. 
Then we can take a holomorphic $k$-vector $Y$ on $U$ such that $\gamma(Y)=s$ as follows. 
We fix a local frame $e$ of $L$ on $U$. 
The contact form $\gamma$ is given by $\gamma=e\otimes \gamma_0$ for a holomorphic 1-form $\gamma_0$ on $U$. 
We can choose a local frame $\{e_1,\dots,e_{2n}\}$ of $D$ and a local section $e_{2n+1}$ of $\mathbb{T}$ on $U$ 
such that $\gamma (e_{2n+1})=e$. 
If $s$ is written by 
$s= \sum s_{i_1 \ldots i_k} e \otimes e_{i_1}\wedge \cdots\wedge e_{i_{k-1}}$
on $U$, then we take a section $Y$ by 
$Y = \sum s_{i_1 \ldots i_k} e_{2n+1}\wedge e_{i_1}\wedge \cdots\wedge e_{i_{k-1}}$. 

Now we consider the smooth section $F(Y)+G(s)$ of $L^{k}\otimes \bigwedge^{k} \mathbb{T}^*$. 
Proposition~\ref{s3p2} implies that $\{F(Y)+G(s)\}|_{D}$ restricted to $D$ is a holomorphic section of $L^k\otimes \bigwedge^k D^*$. 
Hence, there exists a holomorphic section $h$ of $\bigwedge^k D$ such that 
\[
\{F(Y)+G(s)\}|_{D}=F|_{D}(h)
\]
by the isomorphism $F|_{D}$ in Proposition~\ref{s3p1}. 
We define $X$ by $X=Y-h$. 
Then we obtain the holomorphic $k$-vector $X$ on $U$ satisfying the following equations
\begin{eqnarray*}
{\rm (i)}&&\gamma(X)=s, \\
{\rm (ii)}&&\{F(X)+G(s)\}|_{D} =0.
\end{eqnarray*}
We take such local sections $X_1$ and $X_2$ on open sets $U_1$ and $U_2$, respectively. 
The first condition $\rm (i)$ implies that the difference $X_1-X_2$ is in $\bigwedge^k D$ on $U_1\cap U_2$. 
We also have the equation $F|_{D}(X_1-X_2)=F(X_1-X_2)|_{D}=0$ by the second condition $\rm (ii)$. 
Then $X_1=X_2=0$ on $U_1\cap U_2$ since $F|_{D}$ is isomorphic. 
Hence the correspondence of $s$ to $X$ provides a right inverse of 
the map $\gamma : \mathcal{O}(\bigwedge^k \mathbb{T}) \to \mathcal{O}(L\otimes \bigwedge^{k-1} D)$ as a $\mathbb{C}$-module map. 
It induces the splitting of $\mathcal{O}(\bigwedge^k \mathbb{T})$ and the exactness of $H^i(\bigwedge^k \mathbb{T})$ for each $i$. 
It completes the proof. 
\end{proof}

\begin{rem}\label{s3r1}
{\rm 
In the proof, $h$ is independent of the connection $\nabla$ 
since $\{F(Y)+G(s)\}|_{D}$ and $F|_{D}$ do not depend on $\nabla$ by Proposition~\ref{s3p1} and~\ref{s3p2}. 
The $k$-vector $X=Y-h$ is also independent of $\nabla$. 
Hence, the splitting of the sequence in Theorem~\ref{s3t1} is independent of the choice of the connection. 
}
\end{rem}

\subsection{The splitting of the sheaves $\mathcal{O}(L^{m}\otimes\bigwedge^k\mathbb{T})$ as $\mathbb{C}$-module}
In this section, we generalize Theorem~\ref{s3t1} to the splitting of $\mathcal{O}(L^{m}\otimes \bigwedge^k \mathbb{T})$ 
under a condition for $m$ and $k$. 
Let $m$ be an integer such that $m\le -k-1$ or $m\ge -[\frac{k}{2}]$. 
We define a constant $c_{m,i}$ by $c_{m,0}=1$ and 
\[
c_{m,i}=\frac{1}{(k+m)(k+m-1)\cdots(k+m-i+1)\ i!}
\]
for each $i=1,\dots,[\frac{k}{2}]$. 
These constants are well-defined 
since $m+k<0$ and $m+k-[\frac{k}{2}]+1>0$ in the cases of $m\le -k-1$ and $m\ge -[\frac{k}{2}]$, respectively. 
We define a map $F_m : L^{m}\otimes \bigwedge^k \mathbb{T} \to L^{m+k}\otimes \bigwedge^k \mathbb{T}^*$ by 
\[
F_m(X)=\sum_{i=0}^{[\frac{k}{2}]}c_{m,i}\otimes^{k-2i} d^{\nabla}\gamma((d^{\nabla}\gamma)^i(X))\wedge (d^{\nabla}\gamma)^i
\]
for $X\in L^{m}\otimes \bigwedge^k \mathbb{T}$. 
By the same argument in Lemma~\ref{s3l1}, the restriction $F_m|_{D}$ induces the map from $\mathcal{O}(L^{m}\otimes \bigwedge^k D)$ to $\mathcal{O}(L^{m+k}\otimes \bigwedge^k D^*)$. 
 
\begin{prop}\label{s3p3} 
The map $F_m|_{D}:\mathcal{O}(L^{m}\otimes \bigwedge^k D) \to \mathcal{O}(L^{m+k}\otimes \bigwedge^k D^*)$ is isomorphic 
if $k$ and $m$ satisfy one of following three conditions 
\[
\left\{
\begin{array}{ll}
m \le -n-[\frac{k}{2}]-2, &\quad 1\le k\le 2n+1, \\
-n-1 \le m \le -k-1, &\quad 1\le k\le n, \\
m \ge -[\frac{k}{2}], &\quad 1\le k\le 2n+1. 
\end{array}
\right.
\] 
\end{prop}
\begin{proof} 
The map $F_m|_{D}$ is the composition of $\otimes^k d^{\nabla}\gamma|_D$ and the transformation $f_m$ of $\mathcal{O}(L^{m}\otimes \bigwedge^k D)$ 
defined by $f_m(X)=\sum_{i=0}^{[\frac{k}{2}]}c_{m,i}(d^{\nabla}\gamma|_D)^i(X)\wedge w^i$ for $X\in \mathcal{O}(L^{m}\otimes \bigwedge^k D)$. 
At each point $x\in M$, the linear map $(f_m)_x:L^{m}_x\otimes \bigwedge^k D_x\to L^{m}_x\otimes \bigwedge^k D_x$ is written 
by $(f_m)_x=\sum_{i=0}^{ [\frac{k}{2}]} c_{m,i} \mathcal{L}^i\Lambda^i$. 
Proposition~\ref{s2p2} implies that $(f_m)_x$ is isomorphic if $\sum_{s=0}^{r} c_{m,s} \frac{r!}{(r-s)!}\frac{(n-k+r+s)!}{(n-k+r)!}$ is not zero for each $r$. 
If $k\ge 2$, then 
\[
\sum_{s=0}^{r} c_{m,s} \frac{r!}{(r-s)!}\frac{(n-k+r+s)!}{(n-k+r)!}=\frac{(m+n+r+1)\cdots(m+n+2)}{(m+k)\cdots(m+k-r+1)}
\]
for any $r=1,\dots,[\frac{k}{2}]$. 
It yields that $\sum_{s=0}^{r} c_{m,s} \frac{r!}{(r-s)!}\frac{(n-k+r+s)!}{(n-k+r)!}\neq 0$ for each $r$ 
if $m\le -n-[\frac{k}{2}]-2$ or $m \ge -n-1$. 
By the assumption $m\le -k-1$ or $m\ge -[\frac{k}{2}]$, the linear map $(f_m)_x$ is isomorphic 
if 
\[
\left\{
\begin{array}{ll}
m \le -n-[\frac{k}{2}]-2, &\quad \forall k, \\
-n-1 \le m \le -k-1, &\quad 1\le k\le n, \\
m \ge -[\frac{k}{2}], &\quad \forall k. 
\end{array}
\right.
\] 
Hence, we finish the proof. 
\end{proof}

We define a constant $c'_{m,i}$ by
\[
c'_{m,i}=\frac{1}{(k+m)(k+m-1)\cdots(k+m-i)\ i!}
\]
for each $i=0,1,\dots,[\frac{k-1}{2}]$. 
We remark that these constants are well-defined since $m+k<0$ and $m+k-[\frac{k-1}{2}]>0$ in the cases of $m\le -k-1$ and $m\ge -[\frac{k}{2}]$, respectively. 
We define a map $G_m:\Gamma(L^{m+1}\otimes\bigwedge^{k-1} D) \to \Gamma(L^{m+k}\otimes\bigwedge^k \mathbb{T}^*)$ by 
\[
G_m(s)=\sum_{i=0}^{[\frac{k-1}{2}]}c'_{m,i}d^{\nabla}\left( \otimes^{k-1-2i}d^{\nabla}\gamma((d^{\nabla}\gamma)^i(s))\right)\wedge (d^{\nabla}\gamma)^i 
\]
for $s\in \Gamma(L^{m+1}\otimes\bigwedge^{k-1}D)$. 
Similarly to Proposition~\ref{s3p2}, we have 
\begin{prop}\label{s3p4} 
An $L^{k+m}$-valued $k$-form $\{F_m(X)+G_m(\gamma(X))\}|_{D}$ restricted to $D$ is holomorphic for any holomorphic $L^m$-valued $k$-vector $X$. 
$\hfill\Box$
\end{prop}
By repeating the proof of Theorem~\ref{s3t1} with $F_m$ and $G_m$ instead of $F$ and $G$, then we obtain 
\begin{thm}\label{s3t2}
The sequence 
\[
0\to \mathcal{O}(L^{m}\otimes \bigwedge^k D) \to \mathcal{O}(L^{m}\otimes \bigwedge^k \mathbb{T}) \to \mathcal{O}(L^{m+1}\otimes \bigwedge^{k-1} D) \to 0
\] 
splits as sheaves of $\mathbb{C}$-module, and the sequence 
\[
0\to H^i(L^{m}\otimes \bigwedge^k D) \to H^i(L^{m}\otimes \bigwedge^k \mathbb{T}) \to H^i(L^{m+1}\otimes \bigwedge^{k-1} D) \to 0
\] 
is exact for each $i=0,\dots,2n+1$ 
if $k$ and $m$ satisfy one of following three conditions 
\[
\left\{
\begin{array}{ll}
m \le -n-[\frac{k}{2}]-2, &\quad 1\le k\le 2n+1, \\
-n-1 \le m \le -k-1, &\quad 1\le k\le n, \\
m \ge -[\frac{k}{2}], &\quad 1\le k\le 2n+1. \\
\end{array}
\right.
\] 
$\hfill\Box$
\end{thm}

\subsection{The splitting of the sheaves $\mathcal{O}(l^{m}\otimes\bigwedge^k\mathbb{T})$ on $\mathbb{CP}^{2n+1}$}
On the odd dimensional projective space $\mathbb{CP}^{2n+1}$, 
there exists a standard contact structure $\gamma$ written by 
$\gamma =z^0 dz^1 - z^1dz^0 + \cdots + z^{2n}dz^{2n+1} - z^{2n+1}dz^{2n}$ in the homogeneous coordinate. 
Let $l$ denote the hyperplane bundle $\mathcal{O}(1)$ on $\mathbb{CP}^{2n+1}$. 
The associated bundle $L$ is given by $l^2=\mathcal{O}(2)$ and the contact structure $\gamma$ is regarded as a section of $l^2\otimes \mathbb{T}^*$. 

We consider the short exact sequence
\begin{equation*}
0\to l^m \otimes \bigwedge^k D \to l^m \otimes \bigwedge^k \mathbb{T} \stackrel{\gamma}{\to} l^{m+2} \otimes \bigwedge^{k-1}D \to 0 
\end{equation*}
for $m\in \mathbb{Z}$ and $1\le k \le 2n+1$. 
If $m$ is even, then the splitting of $\mathcal{O}(l^{m}\otimes\bigwedge^k\mathbb{T})$ is induced by Theorem~\ref{s3t2} since $L=l^2$. 
From now on, we assume that $m$ is odd. 
We define a constant $\tilde{c}_{m,i}$ by $\tilde{c}_{m,0}=1$ and 
\[
\tilde{c}_{m,i}=\frac{1}{(k+\frac{m}{2})(k+\frac{m}{2}-1)\cdots(k+\frac{m}{2}-i+1)\ i!}
\]
for each $i=1,\dots,[\frac{k}{2}]$. 
These constants are well-defined since $k+\frac{m}{2}-i+1\neq 0$ for any $i$. 
We fix a connection $\nabla$ of $l$ and define a map $\widetilde{F}_m : l^{m}\otimes \bigwedge^k \mathbb{T} \to l^{m+2k}\otimes \bigwedge^k \mathbb{T}^*$ by 
\[
\widetilde{F}_m(X)=\sum_{i=0}^{[\frac{k}{2}]}\tilde{c}_{m,i}\otimes^{k-2i} d^{\nabla}\gamma((d^{\nabla}\gamma)^i(X))\wedge (d^{\nabla}\gamma)^i
\]
for $X\in l^{m}\otimes \bigwedge^k \mathbb{T}$. 
The restriction $\widetilde{F}_m|_{D}$ induces the map from $\mathcal{O}(l^{m}\otimes \bigwedge^k D)$ to $\mathcal{O}(l^{m+2k}\otimes \bigwedge^k D^*)$. 
\begin{prop}\label{s4.2p1} 
The map $\widetilde{F}_m|_{D}:\mathcal{O}(l^{m}\otimes \bigwedge^k D) \to \mathcal{O}(l^{m+2k}\otimes \bigwedge^k D^*)$ is isomorphic 
if $m$ satisfies $m \le -2n-2[\frac{k}{2}]-3$ or $-2n-3 \le m$. 
\end{prop}
\begin{proof} 
If $k\ge 2$, then  
\[
\sum_{s=0}^{r} \tilde{c}_{m,s} \frac{r!}{(r-s)!}\frac{(n-k+r+s)!}{(n-k+r)!}=\frac{(\frac{m}{2}+n+r+1)\cdots(\frac{m}{2}+n+2)}{(\frac{m}{2}+k)\cdots(\frac{m}{2}+k-r+1)}
\]
for any $r=1,\dots,[\frac{k}{2}]$. 
By the same argument in Proposition~\ref{s3p3}, the map $\widetilde{F}_m|_{D}$ is isomorphic 
if $m \le -2n-2[\frac{k}{2}]-3$ or $-2n-3 \le m$. 
\end{proof}

We define a constant $\tilde{c}'_{m,i}$ by
\[
\tilde{c}'_{m,i}=\frac{1}{(k+\frac{m}{2})(k+\frac{m}{2}-1)\cdots(k+\frac{m}{2}-i)\ i!}
\]
for each $i=0,1,\dots,[\frac{k-1}{2}]$. 
These constants are well-defined since $\frac{m}{2}+k-[\frac{k-1}{2}]\neq 0$. 
We define a map $\widetilde{G}_m:\Gamma(l^{m+1}\otimes\bigwedge^{k-1} D) \to \Gamma(l^{m+2k}\otimes\bigwedge^k \mathbb{T}^*)$ by 
\[
\widetilde{G}_m(s)=\sum_{i=0}^{[\frac{k-1}{2}]}\tilde{c}'_{m,i}d^{\nabla}\left( \otimes^{k-1-2i}d^{\nabla}\gamma((d^{\nabla}\gamma)^i(s))\right)\wedge (d^{\nabla}\gamma)^i 
\]
for $s\in \Gamma(l^{m+2}\otimes\bigwedge^{k-1}D)$. 
Then we have 
\begin{prop}\label{s4.2p2} 
An $l^{2k+m}$-valued $k$-form $\{\widetilde{F}_m(X)+\widetilde{G}_m(\gamma(X))\}|_{D}$ restricted to $D$ is holomorphic for any holomorphic $l^m$-valued $k$-vector $X$. 
\end{prop}
\begin{proof} 
We fix a frame $\tilde{e}$ of $l$ on $U$. Then $\gamma$ is given by $\gamma=\tilde{e}^2\otimes \gamma_0$ for a holomorphic $1$-form $\gamma_0$ on $U$. 
Let $X$ be a holomorphic $l^m$-valued $k$-vector $X$ on $M$. 
By the same argument in the proof of Proposition~\ref{s3p2}, we have 
\begin{eqnarray*}
\{\widetilde{F}_m(X)+\widetilde{G}_m(s)\}|_{D}&=&\tilde{e}^{2k+m}\otimes \Big\{ \sum_{i=0}^{[\frac{k}{2}]}\tilde{c}_{m,i}\{ \otimes^{k-2i} d\gamma_0((d\gamma_0)^i(X))\wedge (d\gamma_0)^i \\
&&+\sum_{i=0}^{[\frac{k-1}{2}]}\tilde{c}'_{m,i} d (\otimes^{k-1-2i} d\gamma_0((d\gamma_0)^i(\gamma_0(X))))\wedge (d\gamma_0)^i \Big\} \Big|_D. 
\end{eqnarray*}
Hence $\{\widetilde{F}_m(X)+\widetilde{G}_m(\gamma(X))\}|_{D}$ is holomorphic. 
\end{proof} 
By repeating the proof of Theorem~\ref{s3t1} with $\widetilde{F}_m$ and $\widetilde{G}_m$ instead of $F$ and $G$, then we obtain 
\begin{thm}\label{s4.2t1} 
The sequence 
\[
0\to \mathcal{O}(l^{m}\otimes \bigwedge^k D) \to \mathcal{O}(l^{m}\otimes \bigwedge^k \mathbb{T}) \to \mathcal{O}(l^{m+2}\otimes \bigwedge^{k-1} D) \to 0
\] 
splits as sheaves of $\mathbb{C}$-module, and the sequence 
\[
0\to H^i(l^{m}\otimes \bigwedge^k D) \to H^i(l^{m}\otimes \bigwedge^k \mathbb{T}) \to H^i(l^{m+2}\otimes \bigwedge^{k-1} D) \to 0
\] 
is exact for each $i=0,\dots,2n+1$ if $k$ and $m$ satisfy one of following conditions 
\[
\left\{
\begin{array}{lll}
m \le -2n-2[\frac{k}{2}]-4, &\quad 1\le k\le 2n+1, &\quad m : \text{ even}, \\
-2n-2 \le m \le -2k-2, &\quad 1\le k\le n, &\quad m : \text{ even}, \\
m \ge -2[\frac{k}{2}], &\quad 1\le k\le 2n+1, &\quad m : \text{ even}, \\
m \le -2n-2[\frac{k}{2}]-3, &\quad 1\le k\le 2n+1,&\quad m : \text{ odd}, \\
m\ge -2n-3, &\quad 1\le k\le 2n+1,&\quad m : \text{ odd}. \\
\end{array}
\right.
\] 
$\hfill\Box$
\end{thm}

\section{Vanishing theorems for cohomology of $\bigwedge^k D$}
In this section, we apply the splitting theorems to the cohomology of $\bigwedge^k D$ and obtain the vanishing theorems. 
From now on, we denote by $E \bigwedge^k D$ the tensor $E\otimes \bigwedge^k D$ of vector bundles $E$ and $\bigwedge^k D$ for simplicity. 

\subsection{Vanishing of the cohomology on compact K\"ahler complex contact manifolds}
We have the following vanishing theorem of the cohomology on compact K\"ahler manifolds : 
\begin{thm}\label{s4.1t1} 
If $M$ is a compact K\"ahler complex contact manifold with $c_1(M)>0$, then 
\begin{equation*}
H^i(M, L^m \bigwedge^{k} D)=\{0\} \quad 
\left\{
\begin{array}{lll}
i\le 2n-k,& m \le -[\frac{k+1}{2}]-n-1,& 1\le k\le 2n+1, \\
\forall i,& -n\le m \le -k-1,& 1\le k\le n-1, \\
i\ge k+1,& m \ge -[\frac{k}{2}],&  1\le k\le 2n+1. 
\end{array}
\right.
\end{equation*}
If $c_1(M)<0$, then 
\begin{equation*}
H^i(M, L^m \bigwedge^{k} D)=\{0\} \quad 
\left\{
\begin{array}{lll}
i\ge k+1,& m\le -n-[\frac{k}{2}]-2,&  1\le k\le 2n+1, \\
i\ge k+2,& m=-n-1,& 1\le k\le n, \\
i\le 2n-k-1,& m=-k,& 1\le k\le n, \\
i\le 2n-k,& m\ge -[\frac{k}{2}],&  1\le k\le 2n+1. 
\end{array}
\right.
\end{equation*}
\end{thm}
\begin{proof}
By Theorem~\ref{s3t2}, the sequence 
\begin{equation}
0\to H^i(M, L^m \bigwedge^k D) \to H^i(M, L^m \bigwedge^k \mathbb{T}) \to H^i(M, L^{m+1} \bigwedge^{k-1} D)\to 0 \label{s4.1eq1}
\end{equation}
is exact for each $i$ if $m$ and $k$ satisfy one of following three conditions 
\begin{equation}
\left\{
\begin{array}{ll}
m \le -n-[\frac{k}{2}]-2, &1\le k\le 2n+1, \\
-n-1 \le m \le -k-1, & 1\le k\le n, \ \\
m \ge -[\frac{k}{2}], &1\le k\le 2n+1.  
\end{array}
\right.\label{s4.1eq1.5}
\end{equation}
It follows from Serre's duality that 
\begin{equation}
H^i(M, L^m \bigwedge^k \mathbb{T})^* \cong H^{2n+1-i}(M, \Omega^k L^{-m}K_{M})\cong H^{2n+1-i}(M, \Omega^k L^{-m-n-1}). \label{s4.1eq2}
\end{equation}

In the case of $c_1(M)>0$, 
the first Chern class $c_1(L^{-m-n-1})$ of the line bundle $L^{-m-n-1}$ is negative 
if $m>-n-1$ since $c_1(L)=-\frac{1}{n+1}c_1(K_M)>0$. 
By applying the Kodaira-Akizuki-Nakano vanishing theorem~\cite{AN} to the last cohomology in (\ref{s4.1eq2}), we have 
$H^i(M, L^m \bigwedge^k \mathbb{T})=\{0\}$ for $k+1\le i$ if $m>-n-1$. 
Hence, the sequence (\ref{s4.1eq1}) implies 
\begin{eqnarray}
&&H^i(M, L^m \bigwedge^{k} D)=\{0\}, \label{s4.1eq3} \\ 
&&H^i(M, L^{m+1}\bigwedge^{k-1} D)=\{0\} \label{s4.1eq4}
\end{eqnarray}
for $k+1\le i\le 2n+1$, $m>-n-1$ and $1\le k\le 2n+1$. 
The condition (\ref{s4.1eq4}) is written by $H^i(M, L^{m'}\bigwedge^{k'} D)=\{0\}$ for $k'+2\le i\le 2n+1$, $m'>-n$ and $0\le k'\le 2n$. 
We only consider the vanishing for $k'\ge 1$ 
since the case $k'=0$ that $H^i(M, L^{m'})=\{0\}$ for $2\le i$ and $m'>-n$ is induced by the Kodaira-Akizuki-Nakano vanishing theorem. 
The first condition (\ref{s4.1eq3}) induces the second one (\ref{s4.1eq4}) for $k\ge 1$. 
Hence, (\ref{s4.1eq1.5}) and (\ref{s4.1eq3}) imply that $H^i(M, L^m \bigwedge^{k}D)=\{0\}$ if $m$ satisfies one of two conditions 
\begin{equation}
\left\{
\begin{array}{lll}
i\ge k+1,& -n\le m \le -k-1,& 1\le k \le n-1, \\
i\ge k+1,& m \ge -[\frac{k}{2}],& 1\le k\le 2n+1.  
\end{array}
\right.\label{s4.1eq5.1}
\end{equation}
The Serre's duality implies that 
\begin{equation}
H^i(M, L^m \bigwedge^{k} D)^* \cong H^{2n+1-i}(M, K_{M} L^{-m} \bigwedge^{k} D^*)\cong H^{2n+1-i}(M, L^{-k-m-n-1}\bigwedge^{k}D)\label{s4.1eq6}
\end{equation}
since $D^*=L^{-1}D$. 
We apply the vanishing in (\ref{s4.1eq5.1}) to the last cohomology in (\ref{s4.1eq6}), 
and obtain $H^i(M, L^m \bigwedge^{k}D)=\{0\}$ if $m$ satisfies one of two conditions 
\begin{equation}
\left\{
\begin{array}{lll}
i \le 2n-k,& m \le -[\frac{k+1}{2}]-n-1,&  1\le k\le 2n+1, \\ 
i \le 2n-k,& -n\le m \le -k-1,& 1\le k\le n-1. 
\end{array}
\right.\label{s4.1eq8.1}
\end{equation}
We remark that the condition $k\le -n-1$ implies that $2n-k\ge k+1$. 
It follows from (\ref{s4.1eq5.1}) and (\ref{s4.1eq8.1}) that $H^i(M, L^m \bigwedge^{k} D)=\{0\}$ for 
\begin{equation*}
\left\{
\begin{array}{lll}
i\le 2n-k,& m \le -[\frac{k+1}{2}]-n-1,& 1\le k\le 2n+1, \\
\forall i,& -n\le m \le -k-1,& 1\le k\le n-1, \\
i\ge k+1,& m \ge -[\frac{k}{2}],&  1\le k\le 2n+1. 
\end{array}
\right.
\end{equation*}

In the case of $c_1(M)<0$, the first Chern class $c_1(L^{-m-n-1})$ is negative if $m<-n-1$. 
Then the Kodaira-Akizuki-Nakano vanishing theorem implies 
$H^i(M, L^m \bigwedge^{k} \mathbb{T})=\{0\}$ for $k+1\le i$ if $m<-n-1$. 
It follows from (\ref{s4.1eq1}) that 
\begin{equation}
H^i(M, L^m \bigwedge^{k} D)=\{0\} \quad
\left\{
\begin{array}{lll}
i\ge k+1,& m \le -n-[\frac{k}{2}]-2,& 1\le k\le 2n+1, \\
i\ge k+2,& m=-n-1,& 1\le k\le n. 
\end{array}
\right.\label{s4.1eq13.5}
\end{equation}
Serre's duality implies 
\begin{equation}
H^i(M, L^m \bigwedge^{k} D)=\{0\} \quad
\left\{
\begin{array}{lll}
i\le 2n-k,& m \ge -[\frac{k}{2}],& 1\le k\le 2n+1, \\
i\le 2n-k-1,& m=-k,& 1\le k\le n.  
\end{array}
\right.\label{s4.1eq15}
\end{equation}
Hence it completes the proof. 
\end{proof}

\begin{rem}
{\rm 
Salamon proved that any $(p,q)$-cohomology $H^{p,q}(M)$ vanishes for $p\neq q$ 
if $M$ is the twistor space of a quaternion manifold with a positive scalar curvature~\cite{Sa}. 
In the proof, he also obtained the vanishing of the cohomology of $\bigwedge^k D$ on the twistor space (Equation (6.4) in \cite{Sa}). 
He used the notation of $L$ and $E$ as $L^{\frac{1}{2}}$ and $L^{-\frac{1}{2}}D$ in our notation, respectively. 
These results are improved to the case of compact K\"ahler complex contact manifolds with $c_1(M)>0$ by the same argument, 
and the vanishing is translated into 
\begin{equation}
H^i(M, L^m \bigwedge^{k} D)=\{0\} \quad 
\left\{
\begin{array}{lll}
\forall i,& -n\le m \le -k-1,& 1\le k\le n-1, \\
i\neq k,& m=-k,& 1\le k\le n. 
\end{array}
\right.\label{s4.1eq16}
\end{equation}
The first condition in (\ref{s4.1eq16}) is equal to the second condition in Theorem~\ref{s4.1t1}. 
The second one in (\ref{s4.1eq16}) is independent of our theorem. 
However, we remark that the first and third conditions in Theorem~\ref{s4.1t1} are not induced by the vanishing (\ref{s4.1eq16}). 
}
\end{rem}

\subsection{The vanishing of the cohomology on $\mathbb{CP}^{2n+1}$}
In this section, we show the vanishing theorem for 
$H^i(l^{m}\wedge^k D)=H^i(\mathcal{O}(m)\wedge^k D)$ on $\mathbb{CP}^{2n+1}$ by using Bott's vanishing formula~\cite{Bo}. 
We have the short exact sequence 
\begin{equation*}
0\to D \to \mathbb{T} \stackrel{\gamma}{\to} l^2 \to 0. \label{s4.2eq1}
\end{equation*}
It induces the exact sequence 
\begin{equation}
0\to l^{m} \bigwedge^k D \to l^{m} \bigwedge^k \mathbb{T}\stackrel{\gamma}{\to} l^{m+2} \bigwedge^{k-1} D \to 0 \label{s4.2eq2}
\end{equation}
for $m\in \mathbb{Z}$. 
Then we have a vanishing of the cohomology as follows 
\begin{thm}\label{s4.2t2} 
$H^i(l^{m} \bigwedge^k D)=\{0\}$ if, in the case $m$ is even 
\begin{equation*}
\left\{
\begin{array}{lll}
i\neq 2n+1,&  m\le -2n-2-2[\frac{k+1}{2}], & 1\le k\le 2n+1, \\
i\neq 2n+1-k,&  m= -2n-2, & 1\le k\le n, \\
\forall i,&  -2n\le m\le -2k-2, & 1\le k\le n, \\
i\neq k,&  m= -2k, & 1\le k\le n, \\
i\neq 0,&  m\ge -2[\frac{k}{2}], & 1\le k\le 2n+1 
\end{array}
\right. 
\end{equation*} 
and, in the case $m$ is odd 
\begin{equation*}
\left\{
\begin{array}{lll}
i\neq 2n+1,&  m\le -2n-3-k, & 1\le k\le 2n+1, \\
\forall i,& -2n-2-k\le m\le -2k+1, & 1\le k\le 2n+1, \\
\forall i,& -2n-3\le m\le -k, & 1\le k\le 2n+1, \\
i\neq 0,&  m\le -k+1, & 1\le k\le 2n+1. 
\end{array}
\right. 
\end{equation*} 
\end{thm}
\begin{proof}
By applying Theorem~\ref{s4.2t1} to the sequence (\ref{s4.2eq2}), 
we obtain the short exact sequence
\begin{equation}
0\to H^i(l^{m} \bigwedge^k D) \to H^i(l^{m} \bigwedge^k \mathbb{T}) \to H^i(l^{m+2} \bigwedge^{k-1} D)\to 0 \label{s4.2eq3}
\end{equation}
if 
\begin{equation}
\left\{
\begin{array}{lll}
m \le -2n-2[\tfrac{k}{2}]-4, & 1\le k\le 2n+1, &\quad m : \text{ even}, \\ 
-2n-2 \le m \le -2k-2, & 1\le k\le n, &\quad m : \text{ even}, \\ 
m \ge -2[\tfrac{k}{2}], & 1\le k\le 2n+1, &\quad m : \text{ even}, \\
m \le -2n-2[\tfrac{k}{2}]-3, & 1\le k\le 2n+1,&\quad m : \text{ odd}, \\ 
m \ge -2n-3, & 1\le k\le 2n+1,&\quad m : \text{ odd}. 
\end{array}
\right. \label{s4.2eq3.3}
\end{equation}
By applying Bott's vanishing formula to Serre's duality $H^i(l^{m} \bigwedge ^k \mathbb{T})^* \cong H^{2n+1-i}(\Omega^k (l^{-m-2n-2}))$, 
$H^i(l^{m} \bigwedge ^k \mathbb{T})=\{0\}$ holds except for the following cases 
\begin{equation}
\left\{
\begin{array}{ll}
i=2n+1-k, & m=-2n-2, \\
i=2n+1, & m< -2n-2-k, \\
i=0, & m> -k-1.
\end{array}
\right. \label{s4.2eq5}
\end{equation}
It follows from (\ref{s4.2eq3}) and (\ref{s4.2eq5}) 
that $H^i(l^{m} \bigwedge^k D)=H^i(l^{m+2} \bigwedge^{k-1} D)=\{0\}$
if 
\begin{equation}
\left\{
\begin{array}{ll}
i\neq 2n+1-k, & m=-2n-2, \\
i\neq 2n+1, & m\le -2n-3-k, \\
\forall i, & -2n-2-k \le m\le -k-1,\ (m\neq -2n-2), \\
i\neq 0, & m\ge -k.
\end{array}
\right. \label{s4.2eq6}
\end{equation}
In the case that $m$ is even, (\ref{s4.2eq3.3}) and (\ref{s4.2eq6}) imply that 
$H^i(l^{m} \bigwedge^k D)=\{0\}$ and $H^i(l^{m+2} \bigwedge^{k-1} D)=\{0\}$ if 
\begin{equation}
\left\{
\begin{array}{lll}
i\neq 2n+1-k,&  m= -2n-2, & 1\le k\le n, \\
i\neq 2n+1,&  m\le -2n-4-2[\frac{k}{2}], & 1\le k\le 2n+1, \\
\forall i,&  -2n-1\le m\le -2k-2, & 1\le k\le n, \\
i\neq 0,&  m\ge -2[\frac{k}{2}], & 1\le k\le 2n+1. 
\end{array}
\right. \label{s4.2eq9}
\end{equation} 
We replace the vanishing $H^i(l^{m+2} \bigwedge^{k-1} D)=\{0\}$ for (\ref{s4.2eq9}) by $H^i(l^{m} \bigwedge^k D)=\{0\}$ 
for 
\begin{equation}
\left\{
\begin{array}{lll}
i\neq 2n-k,&  m= -2n, & 0\le k\le n-1, \\
i\neq 2n+1,&  m\le -2n-2-2[\frac{k+1}{2}], & 0\le k\le 2n, \\
\forall i,&  -2n+1\le m\le -2k-2, & 0\le k\le n-1, \\
i\neq 0,&  m\ge -2[\frac{k+1}{2}]+2, & 0\le k\le 2n. 
\end{array}
\right. \label{s4.2eq10}
\end{equation} 
We remark that the case $k=0$ of (\ref{s4.2eq10}) is contained in Bott's vanishing formula. 
We only consider the vanishing of $H^i(l^{m} \bigwedge^k D)$ for $k\ge 1$. 
We summarize (\ref{s4.2eq9}) and (\ref{s4.2eq10}) as $H^i(l^{m} \bigwedge^k D)=\{0\}$ for even $m$ and 
\begin{equation}
\left\{
\begin{array}{lll}
i\neq 2n+1,&  m\le -2n-2-2[\frac{k+1}{2}], & 1\le k\le 2n+1, \\
i\neq 2n+1-k,&  m= -2n-2, & 1\le k\le n, \\
\forall i,&  -2n\le m\le -2k-2, & 1\le k\le n, \\
i\neq 0,&  m\ge -2[\frac{k}{2}], & 1\le k\le 2n+1. 
\end{array}
\right. \label{s4.2eq11}
\end{equation} 
In the case that $m$ is odd, by repeating the above argument the conditions (\ref{s4.2eq3.3}) and (\ref{s4.2eq6}) imply that 
$H^i(l^{m} \bigwedge^k D)=\{0\}$ for 
\begin{equation}
\left\{
\begin{array}{lll}
i\neq 2n+1,&  m\le -2n-3-k, & 1\le k\le 2n+1, \\
\forall i,&  m= -2n-2-k, & 1\le k\le 2n+1, \\
\forall i,&  m= -2n-1-k, & 1\le k\le 2n, \\
\forall i,&  -2n-3\le m\le -k-1, & 1\le k\le 2n+1, \\
\forall i,&  m=-k, & 1\le k\le 2n-1, \\
i\neq 0,&  m= -k, & k= 2n+1, \\
i\neq 0,&  m\ge -k+1, & 1\le k\le 2n+1. 
\end{array}
\right. \label{s4.2eq18}
\end{equation} 
Applying (\ref{s4.2eq11}) and (\ref{s4.2eq18}) to Serre's duality 
$H^i(l^{m} \bigwedge^k D)^*\cong H^{2n+1-i}(l^{-m-2n-2-2k} \bigwedge^k D)$, 
we obtain $H^i(l^{m} \bigwedge^k D)=\{0\}$ if, in the case $m$ is even 
\begin{equation}
\left\{
\begin{array}{lll}
i\neq 0,&  m\ge -2[\frac{k}{2}], & 1\le k\le 2n+1, \\
i\neq k,&  m= -2k, & 1\le k\le n, \\
\forall i,&  -2n\le m\le -2k-2, & 1\le k\le n, \\
i\neq 2n+1,&  m\le -2n-2-2[\frac{k+1}{2}], & 1\le k\le 2n+1 
\end{array}
\right. \label{s4.2eq20}
\end{equation}
and, in the case $m$ is odd 
\begin{equation}
\left\{
\begin{array}{lll}
i\neq 0,&  m\ge -k+1, & 1\le k\le 2n+1, \\
\forall i,&  m= -k, & 1\le k\le 2n+1, \\
\forall i,&  m= -k-1, & 1\le k\le 2n, \\
\forall i,&  -2n-1-k\le m\le -2k+1, & 1\le k\le 2n+1, \\
\forall i,&  m=-2n-2-k, & 1\le k\le 2n-1, \\
i\neq 2n+1,&  m= -2n-2-k, & k= 2n+1, \\
i\neq 2n+1,&  m\le -2n-3-k, & 1\le k\le 2n+1. 
\end{array}
\right. \label{s4.2eq21}
\end{equation}
The conditions (\ref{s4.2eq11}) and (\ref{s4.2eq20}) in the case of even $m$, 
and (\ref{s4.2eq18}) and (\ref{s4.2eq21}) in the case of odd $m$ induce the conditions in theorem. 
Hence it completes the proof. 
\end{proof}
\begin{rem}
{\rm 
A contact structure on $\mathbb{CP}^{2n+1}$ is unique, up to automorphisms~\cite{NT}. 
Hence Theorem~\ref{s4.2t1} and~\ref{s4.2t2} hold for any contact structure on $\mathbb{CP}^{2n+1}$. 
}
\end{rem}

\begin{rem}
{\rm 
In algebraic geometry, the {\it null correlation bundle} $N$ on $\mathbb{CP}^{2n+1}$ is defined by the short exact sequence 
\[
0\to N \to \mathbb{T}(-1)\to \mathcal{O}(1) \to 0
\]
where $\mathbb{T}(-1)=\mathcal{O}(-1)\otimes \mathbb{T}$ (\cite{OSS}). 
It induces to the following :
\[
0\to N(1)\to \mathbb{T} \to \mathcal{O}(2) \to 0.
\]
The bundle $N(1)=\mathcal{O}(1)\otimes N$ is the kernel of a transformation $A: \mathbb{T} \to \mathcal{O}(2)$ 
which is given by $a=\sum a_{ij}z_idz_j$ in the homogeneous coordinate.  
If $a_{ij} = -a_{ji}$ and $(a_{ij})$ is non-degenerate, then $A$ induces a contact structure on $\mathbb{CP}^{2n+1}$. 
By applying Theorem~\ref{s4.2t2} to $D= N(1)$, 
we obtain the vanishing formula for the cohomology $H^i(\wedge^k N(m+k))$. 
}
\end{rem}

\noindent
\textbf{Acknowledgements}. 
The authors would like to thank Professor S. Nayatani 
for his useful comments and advice. 
The first named author is supported by Grant-in-Aid for Young Scientists (B) $\sharp$17K14187 from JSPS.

\vspace{\baselineskip}
\begin{flushright}
\begin{tabular}{l}
\textsc{Takayuki Moriyama}\\
Department of Mathematics\\
Mie University\\
Mie 514-8507, Japan\\
E-mail: takayuki@edu.mie-u.ac.jp\\
\\

\textsc{Takashi Nitta}\\
Department of Mathematics\\
Mie University\\
Mie 514-8507, Japan\\
E-mail: nitta@edu.mie-u.ac.jp
\end{tabular}
\end{flushright}

\end{document}